\documentclass[12pt,a4paper]{article}
\usepackage{mathrsfs}
\usepackage{amsfonts}
\usepackage{cases}
\usepackage{times}\usepackage{latexsym}
\usepackage{graphics}\usepackage{epsfig}
\usepackage{color}
\usepackage{amsthm}\usepackage{setspace}\usepackage{graphicx}\usepackage{float}\usepackage{subfigure}\usepackage{dsfont}\usepackage{fancyhdr}\usepackage{amssymb}
\newenvironment{keywords}{{\bf Keywords: }}{}

\newtheorem{lemma}{Lemma}[section]

\newtheorem{theorem}{Theorem}[section]

\begin{document}

\author{Hanbing Liu $^1$ \ \ Peng Hu$^1$ \ \  Ionu\c{t} Munteanu$^2$\\ {\small $^1$ School of Mathematics and Physics of China University of
Geoscience,} \\{\small Wuhan, 430074, P.R.China. E-mail:
hanbing272003@aliyun.com}\\ {\small $^2$ University Al. I. Cuza, Faculty of Mathematics} \\ {\small and Octav Mayer Institute of Mathematics (Romanian Academy), Ia\c{s}i(700506), Romania}}
\title
 {Boundary feedback stabilization of Fisher's equation}
 \date{}
 \maketitle

%----------Author 2
%\author{A Second Author}
%\address{The address of\br
%the second author\br sitting somewhere\br in the world}
%\email{dont@know.who.knows}

 \hrule
\begin{abstract}
The aim of this work is to design an explicit finite dimensional boundary feedback controller for locally exponentially stabilizing the equilibrium solutions to Fisher's equation in both $L^2(0,1)$ and $H^1(0,1)$. The feedback controller is expressed in terms of the eigenfunctions corresponding to unstable eigenvalues of the linearized equation. This stabilizing procedure is applicable for any level of instability, which extends the result of \cite{02} for nonlinear parabolic equations. The effectiveness of the approach is illustrated by a numerical simulation.
\end{abstract}

 \noindent
% -----------------------------------------------------------------------------

\section{Introduction}
\label{}
We consider the following Fisher's equation
\begin{equation}\label{e101}
 \left\{\begin{array}{ll}
 u_t(x,t)- u_{xx}(x,t)-\alpha u(x,t)+\beta u^2(x,t)=0, \\
 \   \ \ \ \ (x,t)\in (0,1)\times (0,\infty),\\
 u(0,t)=0, u(1,t)=U(t),\ \ \ \ \ \ \ \   \ t\in (0,\infty),\\
  u(x,0)=u_0(x),\ \ \ \ \ \ \ \  \ \ \ x\in(0,1).
\end{array}\right.
\end{equation}
where $\alpha, \beta$ are positive constants; $u(x,t)$ represents the state evolution over the spatial-temporal domain characterized by the coordinates $x, t$, respectively. The scalar function $U(t)$ represents the boundary actuation via the Dirichlet boundary condition. Our object is to stabilize the zero equilibrium with finite dimensional boundary feedback controller.

There has been a large body of remarkable results in recent years on boundary stabilization of linear and nonlinear parabolic systems, and we cite here \cite{01}-- \cite{12}.  The existence of a stabilizing linear boundary feedback controller for the linear parabolic equations with Dirichlet or Neumann boundary conditions was established firstly by R. Triggiani in his pioneering work \cite{01}. The idea was to decompose the controlled system in a finite dimensional unstable part corresponding to unstable eigenvalues and an infinite dimensional stable part. Using Kalman's theory for the finite dimensional part, it was shown in \cite{01} that boundary stabilization is possible via a feedback controller with finite dimensional structure, but only conceptual procedure are provided, instead of an explicit form. V. Barbu firstly introduced a new technique in his work \cite{02, 03, 04} for the construction of a feedback controller, which is finite dimensional and in an explicit form. However, his work based on an strict assumption which requires the normal derivative of the eigenfunctions corresponding to the unstable eigenvalues are linearly independent. This assumption limits the range of application, and it holds true for 1-D heat equation only for low levels of instability.  By following the same approach, I. Munteanu provided explicit stabilizing boundary feedback controllers for the phase-field system in \cite{06}, for the heat equation with fading memory in \cite{07}, and also, for the Navier-Stokes equations with fading memory in \cite{08}. Later on, based on V. Barbu's work, I. Munteanu developed a delicate approach in \cite{09} to the boundary feedback controller which stabilizes the semilinear parabolic equations in $L^2$, and this method dropped the above assumption. In last decade, a different approach, which is the so-called Backstepping technique is developed, and it works efficiently for the 1-D linear parabolic equation on a finite rod, of any level of instability (\cite{10}-\cite{12}).

Recently, X. Yu \emph{et al.} studied the local boundary feedback stabilizing of 1-D Fisher's equation by Backstepping method (\cite{13}), wherein it is proved the feedback controller stabilizes the system in both $L^2(0,1)$ and $H^1(0,1)$, and numerical examples are provided to illustrate the effectiveness. Fisher's equation is a nonlinear parabolic equation firstly proposed by Fisher to model the advance of a mutant gene in an infinite one-dimensional habitat \cite{14}. Moreover, Fisher¡¯s equation has been used as a basis
for a wide variety of models for the spatial spread of gene in population, chemical wave propagation, flame propagation, branching Brownian motion process and even nuclear reactor theory \cite{15}-\cite{18}. It is well known that the uncontrolled Fisher¡¯s equation is unstable.

In this work, for stabilizing the Fisher's equation, we shall adopt the technique in \cite{09} to design a feedback controller, which is different from the backstepping method applied in \cite{13}, but the form is much simpler. In the next section, by following similar arguments as in \cite{09}, we shall prove that the designed finite dimensional boundary feedback controller globally stabilizes the linearized equation
 \begin{equation}\label{e102}
 \left\{\begin{array}{ll}
 u_t(x,t)-u_{xx}(x,t)-\alpha u(x,t)=0,\\ \   \ \ \ \ (x,t)\in (0,1)\times (0,\infty),\\
 u(0,t)=0, u(1,t)=U(t)\ \ \ \ \  \ \ \ \ \ t\in (0,\infty),\\
  u(x,0)=u_0(x)\ \ \ \ \ \ \   \ \ \ x\in(0,1).
\end{array}\right.
\end{equation}
It should be mentioned that it was shown in \cite{09} only the stabilization result of the linearized equation, but for the stabilization of the nonlinear equation, the author referred to \cite{02}. However, we find that, since the structure of the feedback controller is different, the idea to prove the nonlinear stabilization result in Theorem 4.1 in \cite{02} can't be applied directly in our case. In Section 3, we shall prove in detail that this feedback controller locally stabilizes Fisher's equation around zero in $L^2(0,1)$. Hence, for the first time, this work extends completely the work of V.Barbu in \cite{02} for nonlinear parabolic equations to any level of instability. Furthermore, it will be shown that the solution to the closed-loop system is locally exponentially stable in $H^1(0,1)$ when the initial data is in $H^1(0,1)$ and satisfies certain compatibility condition. In Section 4,  Numerical simulations of the closed-loop system will be provided to illustrate the effectiveness of our algorithm. We want to stress that, the method applied in this paper works in all dimensions, although we only considered in this work the 1-D case for the purpose to focus on the idea of algorithm.

\section{Stabilization of the linear equation}
\setcounter{equation}{0}
Define operator $\mathcal{A}:D(\mathcal{A})\rightarrow L^2(0,1), D(\mathcal{A})= H^2(0,1)\cap H^1_0(0,1)$ as
\begin{equation}\label{e103}
\mathcal{A} u=-\partial_{xx} u-\alpha u.
\end{equation}
Notice that $\mathcal{A}$ is self-adjoint. Besides, it can be shown that $\mathcal{A}$ has compact resolvent. Therefore, it has countable set of eigenvalues, denoted by $\{\lambda_j\}_{j=1}^{\infty}$. In fact, by solving the equations
\begin{equation}\label{e104}
 \left\{\begin{array}{ll}
(\lambda_j+\alpha) u(x)-u''(x)=0, \   \ \ \ \ x\in (0,1),\\
 u(0)= u(1)=0,\ \ \ \ \
\end{array}\right.
\end{equation}
$j=1,2, \cdots$, we can find that
\begin{equation}\label{e105}
\lambda_j=(\pi j)^2-\alpha,
\end{equation}
and the corresponding eigenfunctions to $\lambda_j$ is
\begin{equation}\label{e106}
\phi_j(x)=\sin(\pi jx), x\in(0,1).
\end{equation}
Given $\rho>0$, there exists only a finite number of eigenvalues $\{\lambda_j\}_{j=1}^{N}$ with $\lambda_j<\rho, j=1, 2, \cdots, N$, and $\lambda_j\geq \rho$ for all $j\geq N+1$. As we shall see below, the larger $\rho$ is, the faster of the solution to the controlled system exponentially decay is, but on the other hand, the lager dimension, $N$, of the feedback controller is.

In the rest of the paper, we shall denote by $\|\cdot\|$ and $\|\cdot\|_1$ the norms of $L^2(0,1)$ and $H^1(0,1)$ respectively. The inner products in $L^2(0,1)$ and Euclid space $\mathds{R}^N$ will be denoted by $\langle \cdot \rangle$ and  $\langle \cdot \rangle_N$, respectively. We shall write $Q=(0,1)\times (0,\infty)$ for simplicity, and the variables $x,t$ will be omitted in the case of no ambiguity.

Let us denote by $B_0$ the Gram matrix of system $\{\phi_j'(1)\}_{1\leq j\leq N}$, i.e.

%%\begin{equation}\label{e201}
$$B_0:=\left(\begin{array}{cccc}\phi_1'(1)\cdot\phi_1'(1) & \phi_1'(1)\cdot\phi_2'(1)\cdots &  \phi_1'(1)\cdot\phi_N'(1)\\ \phi_2'(1)\cdot\phi_1'(1) & \phi_2'(1)\cdot\phi_2'(1)\cdots &  \phi_2'(1)\cdot\phi_N'(1) \\  \cdots & \cdots & \cdots\\ \phi_N'(1)\cdot\phi_1'(1) & \phi_N'(1)\cdot\phi_2'(1)\cdots &  \phi_N'(1)\cdot\phi_N'(1)\end{array}\right).$$
%%\end{equation}

Now, let $\rho<\gamma_1<\gamma_2<\cdots<\gamma_N$ be $N$ constants, which are sufficiently large, such that, for each $j\in\{1, 2, \cdots, N\}$, the solution to the equation
\begin{equation}\label{e202}
\left\{\begin{array}{ll}
\gamma_j\psi_j+\psi''_j-\alpha\psi_j-2\sum_{k=1}^N\lambda_k\langle\psi_j, \phi_k\rangle\phi_k=0,\\ \ \ \ \mathrm{in} \ (0,1),\\
\psi_j(0)=0, \psi_j(1)=V,
\end{array}\right.
\end{equation}
exists for any given $V\in \mathds{R}$. We shall denote by $D_{\gamma_j}$ the map $:V\rightarrow \psi_j(\cdot)$, i.e., $\psi_j(\cdot)=D_{\gamma_j} V$. It is well known that $\psi_j\in H^{1/2}(0,1)$ and $\|\psi_j\|_{1/2}\leq C|V|$ (this result holds also in multidimensional case). Here we denote by the $\|\cdot\|_{1/2}$ the norm of the space $H^{1/2}(0,1)$.

We introduce the matrices
\begin{equation}\label{e203}
\Lambda_{\gamma_k}:=\mathrm{diag}(\frac{1}{\gamma_k-\lambda_i})_{1\leq i\leq N}, k=1, 2, \cdots, N,
\end{equation}
and
\begin{equation}\label{e205}
B=(B_1+B_2+\cdots+B_N)^{-1},
\end{equation}
where $B_k=\Lambda_{\gamma_k}B_0\Lambda_{\gamma_k}, k=1,2, \cdots, N$. It will be shown in Appendix that $B_1+B_2+\cdots+B_N$ is invertible (see also \cite{09}).

Now, let us introduce the feedbacks
\begin{eqnarray}\label{e206}
&&U_k(t)=F_k(u)(t)\nonumber\\
&&:=\langle B\left(\begin{array}{ccc}\int_0^1u(t,x)\phi_1(x)dx\\ \int_0^1u(t,x)\phi_2(x)dx\\ \cdots\\ \int_0^1u(t,x)\phi_N(x)dx\end{array}\right), \left(\begin{array}{ccc}\frac{1}{\gamma_k-\lambda_1}\phi_1'(1)\\ \frac{1}{\gamma_k-\lambda_2}\phi_2'(1)\\ \cdots\\ \frac{1}{\gamma_k-\lambda_N}\phi_N'(1)\end{array}\right)\rangle_N,
\end{eqnarray}
$t>0, k=1, 2, \cdots, N$. In the case where the Gram matrix $B_0$ is nonsingular, one may consider only one feedback form from the above list, for example, the first one. In this case, in fact, we stumble on the feedback designed in \cite{02}.

When $B_0$ is not invertible, we take $U$ of the form
\begin{equation}\label{e207}
U(t)=F(u)(t):=\sum_{k=1}^NF_k(u)(t)=\sum_{k=1}^NU_k(t).
\end{equation}
We can see that the feedback controller here only involves some matrices related to the first $N$ eigenfunctions, while the feedback controller designed in \cite{13} involves Bessel function which is more complex. We claim that the above feedback assures the stability of the steady-state $0$ of (\ref{e102}), and then stabilizes the nonlinear system (\ref{e101}) locally around the zero solution.

The following result amounts to saying that the feedback $U(t)$ achieves global exponential stability of the linear system  (\ref{e102}). More precisely,
\begin{theorem}\label{th201}
Assume that $u_0\in L^2(0,1)$. The feedback $U(t)$, given by (\ref{e207}), exponentially stabilizes the linearized equation (\ref{e102}). More exactly, there exist constants $C>0, \mu>0$, such that the solution to equation (\ref{e102}) with the control in (\ref{e207}), satisfies
\begin{equation}\label{e210}
\|y(t)\|\leq Ce^{-\mu t}\|y_0\|, t\geq 0.
\end{equation}
\end{theorem}

\begin{proof}
The arguments are similar to that of Theorem 2.1 in \cite{09}. However, for the later use, we still briefly give the proof here.
Denote by $h_j$ the solution to equation (\ref{e202}) with boundary condition $h_j(1,t)=U_j(t), 1\leq j\leq N$, i.e., $h_j(t,\cdot)=D_{\gamma_j}U_j(t)$. Write $v(t,x)=u(t,x)-\sum_{j=1}^N h_j(t,x)$, then $v(t,0)=v(t, 1)=0, \forall t>0$, and
\begin{equation}\label{e211}
 \left\{\begin{array}{ll}
 v_t+\mathcal{A}v(t)=R(h_1, h_2, \cdots, h_N), \mathrm{in}\ Q,\\
 v(0,x)=v_0(x),\ \ \ \ \ \ \  \ \  \  \ \  x\in(0,1),
\end{array}\right.
\end{equation}
Here
\begin{eqnarray}\label{e212}
&&R(h_1, h_2, \cdots, h_N)=-\sum_{i=1}^Nh_i'(t)\nonumber\\
&&-2\sum_{k=1, i=1}^N\lambda_k\langle h_i(t), \phi_k\rangle\phi_k+\sum_{i=1}^N\gamma_ih_i(t).
\end{eqnarray}
We shall write this term for simplicity by $R$ in the following.

Involving equations (\ref{e104}) and (\ref{e202}), by simple calculation, we can get that, for $1\leq i, k\leq N$,
\begin{equation}\label{e213}
\langle h_k, \phi_i\rangle=-\frac{1}{\gamma-\lambda_i}U_k(t)\cdot \phi'_i(1).
\end{equation}
With this identity, and the definition of $U_k(t)$, it follows that
\begin{equation}\label{e214}
\left(\begin{array}{ccc}\langle h_k, \phi_1\rangle\\ \langle h_k, \phi_2\rangle\\ \cdots\\ \langle h_k, \phi_N\rangle\end{array}\right)=-B_kB\left(\begin{array}{ccc}\langle u(t), \phi_1\rangle\\ \langle u(t), \phi_2\rangle\\ \cdots\\ \langle u(t), \phi_N\rangle\end{array}\right).
\end{equation}
By the latter equation, and the relation between $v$ and $u$, one can obtain that
\begin{eqnarray}\label{e215}
&&U_k(t)=\tilde{F}_k(v)(t)\nonumber\\&&=\frac{1}{2}\langle B\left(\begin{array}{ccc}\int_0^1v(t,x)\phi_1(x)dx\\ \int_0^1v(t,x)\phi_2(x)dx\\ \cdots\\ \int_0^1v(t,x)\phi_N(x)dx\end{array}\right), \left(\begin{array}{ccc}\frac{1}{\gamma_k-\lambda_1}\phi_1'(1)\\ \frac{1}{\gamma_k-\lambda_2}\phi_2'(1)\\ \cdots\\ \frac{1}{\gamma_k-\lambda_N}\phi_N'(1)\end{array}\right)\rangle_N.
\end{eqnarray}
Moreover, likewise in (\ref{e214}), we have, for $k=1, 2, \cdots, N$,
 \begin{equation}\label{e216}
\left(\begin{array}{ccc}\langle h_k, \phi_1\rangle\\ \langle h_k, \phi_2\rangle\\ \cdots\\ \langle h_k, \phi_N\rangle\end{array}\right)=-\frac{1}{2}B_kB\left(\begin{array}{ccc}\langle v(t), \phi_1\rangle\\ \langle v(t), \phi_2\rangle\\ \cdots\\ \langle v(t), \phi_N\rangle\end{array}\right).
\end{equation}

Denote by $\mathbf{v}^N$ the vector consisted by the first $N$ modes of $v$, i.e., $$\mathbf{v}^N(t)=(\langle v(t), \phi_1\rangle,\langle v(t), \phi_2\rangle, \cdots, \langle v(t), \phi_N\rangle)^T,$$ and denote by $\Lambda$ the diagonal matrix $\mathrm{diag}(\lambda_i)_{1\leq i\leq N}$. Multiplying equation (\ref{e211}) by $\phi_1, \phi_2, \cdots, \phi_N$ respectively, and using identity (\ref{e216}), one can obtain the equation satisfied by $\mathbf{v}^N$ as follows,
\begin{equation}\label{e217}
 \left\{\begin{array}{ll}
 \frac{d}{dt}\mathbf{v}^N(t)=-\gamma_1\mathbf{v}^N(t)+\sum_{k=2}^N(\gamma_1-\gamma_k)B_kB\mathbf{v}^N(t), t>0,\\
 \mathbf{v}^N(0)=\mathbf{v}^N_0.
\end{array}\right.
\end{equation}
Recall that $B_j, 1\leq j\leq N$ are positive semidefinite symmetric matrices, we know that, $B=(B_1+B_2+\cdots+B_N)^{-1}$ is a positive definite symmetric matrix. Thus, one can define another positive definite symmetric matrix, denoted by $B^{1/2}$, such that $B^{\frac{1}{2}}B^{\frac{1}{2}}=B$. Multiplying Equation (\ref{e217}) by $B\mathbf{v}^N$, we get that
\begin{eqnarray}\label{e218}
&&\frac{1}{2}\frac{d}{dt}|B^{\frac{1}{2}}\mathbf{v}^N(t)|_N=-\gamma_1|B^{\frac{1}{2}}\mathbf{v}^N(t)|_N\nonumber\\
&&+\sum_{k=2}^N(\gamma_1-\gamma_k)\langle B_kB\mathbf{v}^N, B\mathbf{v}^N \rangle_N.
 \end{eqnarray}
Since $\gamma_1-\gamma_k<0, k=2, 3, \cdots, N$, this leads to
\begin{equation}\label{e219}
|\frac{d}{dt}\mathbf{v}^N(t)|_N+|\mathbf{v}^N(t)|_N\leq e^{-2\gamma_1 t}|\mathbf{v}^N_0|_N, \forall t>0.
 \end{equation}

Write $X_1=\mathrm{linspan}\{\phi_j\}_{j=1}^N, X_2=X_1^{\perp}, P_N$ the algebraic projection $L^2(0,1)$ onto $X_1$, and set $v_1=P_N v, v_2=(I-P_N)v$. Then, $v=v_1+v_2$, and we may decompose equation (\ref{e211})
into the following system
\begin{equation}\label{e220}
 \left\{\begin{array}{ll}
 \frac{dv_1}{dt}+\mathcal{A}_1v_1(t)=P_NR(h_1, h_2, \cdots, h_N), t> 0,\\
  \frac{dv_2}{dt}+\mathcal{A}_2v_2(t)=(I-P_N)R(h_1, h_2, \cdots, h_N),t> 0,
\end{array}\right.
\end{equation}
where $\mathcal{A}_1=P_N\mathcal{A}$, $\mathcal{A}_2=(I-P_N)\mathcal{A}$.

We can refer from equation (\ref{e219}) that
\begin{equation}\label{e221}
\|v_1(t)\|\leq Ce^{-2\gamma_1 t}\|v_1(0)\|, \forall t>0.
 \end{equation}
Moreover, it follows by equation (\ref{e219}), the definition of $h_k$ and identity (\ref{e215}) that
\begin{equation}\label{e222}
\|h_k(t)\|+\|h_k'(t)\|\leq e^{-2\gamma_1 t}\|v_1(0)\|, \forall t>0.
\end{equation}
One can obtain from the second equation in (\ref{e220}) that
\begin{eqnarray}\label{e223}
 &&\|v_2(t)\|\leq \|e^{-\mathcal{A}_2t}v_2(0)+\int_0^te^{-\mathcal{A}_2(t-s)}(I-P_N)Rds\|\nonumber\\
 &&\leq Ce^{-\rho t}\|v_0\|, \ t>0.
\end{eqnarray}
Finally, by (\ref{e221}), (\ref{e223}) and the relation between $u$ and $v$, we obtain the estimate
\begin{eqnarray}\label{e224}
\|u(t)\|\leq C e^{-\mu t}\|u_0\|, t\geq 0,
\end{eqnarray}
for some $\mu>0$. This completes the proof of Theorem \ref{th201}.
\end{proof}

%\begin{remark}\label{re201}
%In addition to the above results, one may show that the solution $v$ to (\ref{e211}) satisfies
%\begin{equation}\label{e225}
%\int_{0}^{\infty}\|v(t)\|^2_1dt\leq C\|v(0)\|^2,
%\end{equation}
%for some $C>0$. Indeed, for $j=1, 2, \cdots, N$,
%$$\int_{0}^{\infty}\lambda_j|\langle v(t), \phi_j\rangle|^2dt\leq \lambda_N \frac{C}{\gamma_1}\|v_1(0)\|^2.$$
%For $j\geq N+1$,
%$$\int_{0}^{\infty}\lambda_j|\langle v(t), \phi_j\rangle|^2dt\leq \int_{0}^{\infty}\lambda_j|e^{-\lambda_jt}\langle v(0), \phi_j\rangle|^2dt$$
%$$+\int_{0}^{\infty}\int_0^t\lambda_j|e^{-\lambda_j(t-\tau)}\langle (I-P_N)R(h_1, h_2, \cdots, h_N), \phi_j\rangle|^2d\tau dt$$
%$$\leq C\|v(0)\|^2.$$
%The above inequalities yield (\ref{e225}).
%\end{remark}
\section{Stabilization of nonlinear equation}
\setcounter{equation}{0}
For the stabilization of the nonlinear Fisher's equation, we have the following result.
\begin{theorem}\label{th301}
Assume that $u_0\in L^2(0,1)$. The feedback $U$, given by (\ref{e207}), locally stabilizes the equation (\ref{e101}). More exactly, there exist constants $C>0, \mu>0$ and $\delta>0$, such that for all $\|u_0\|\leq \delta$, the solution to equation (\ref{e101}) with control in (\ref{e207}), satisfies
\begin{equation}\label{e302}
\|u(t)\|\leq Ce^{-\mu t}\|u_0\|, t\geq 0.
\end{equation}
\end{theorem}

Since the trace of function in space $H^1(0,1)$ is well-defined, to study the stabilization result in the space $H^1(0,1)$, we need to give some compatibility condition of the boundary control and the initial data. More exactly, the initial data need to be taken from the following space
$$I_C=\{u\in H^1(0,1); F(u)=u(1), u(0)=0\}.$$
\begin{theorem}\label{th302}
Assume that $u_0\in I_C$. The feedback $U$, given by (\ref{e207}), locally stabilizes the equation (\ref{e101}) in $H^1(0,1)$. More exactly, there exist constants $C>0, \mu>0$ and $\delta>0$, such that for all $\|u_0\|\leq \delta$, the solution to equation (\ref{e101}) with control in (\ref{e207}), satisfies
\begin{equation}\label{e304}
\|u(t)\|_1\leq Ce^{-\mu t}\|u_0\|_1, t\geq 0.
\end{equation}
\end{theorem}
We mention here that the proof of Theorem \ref{th301} is not straightforward by the proof of Theorem 4.1 in \cite{02} because the structure of the feedback controller is different. More concretely, we should find a different way from that in \cite{02} to combine the decomposited systems as a semigroup form, so we shall show this part in detail, and the rest follows directly. Theorem \ref{th302} is new, we shall also give the proof below.
\begin{proof}[Proof of Theorem \ref{th301}]
We still ``lift'' the boundary condition firstly. Denote by $v=u-\sum_{k=1}^Nh_k$. Here, $h_k=D_{\gamma_k}U_k$, and $U_k, 1\leq k\leq N$ are defined as in (\ref{e206}). We know by (\ref{e215}) that $h_k=D_{\gamma_k}\tilde{F}_k(v)$. Denote $\tilde{D}_F(v)=\sum_{k=1}^ND_{\gamma_k}\tilde{F}_k(v)$. Then, $u=(I+\tilde{D}_F)v$. It suffices to prove the existence of exponentially stable solution to the following equation
\begin{equation}\label{e305}
 \left\{\begin{array}{ll}
v_t- v_{xx}-\alpha v+\beta ((I+\tilde{D}_F)v)^2=R, \mathrm{in} \ Q,\\
 v(0,t)=0, v(1,t)=0,\ \ \ \ \ \ \ \ t\in (0,\infty),\\
  v(x,0)=v_0(x)\ \ \ \ \ \ \ x\in(0,1).
\end{array}\right.
\end{equation}
where $R$ is the same as in (\ref{e212}). Write $h_k$ in the term $R$ by $D_{\gamma_k}\tilde{F}_k(v)$ for $1\leq k\leq N$, then equation (\ref{e305}) can be written in an abstract form as
\begin{eqnarray}
&&(I+\tilde{D}_F)v_t=\mathcal{A}v+2\sum_{k,i=1}^N\lambda_k\langle D_{\gamma_i}\tilde{F}_i(v), \phi_k\rangle\phi_k\nonumber\\
&&-\sum_{i=1}^N\gamma_iD_{\gamma_i}\tilde{F}_iv+\beta ((I+\tilde{D}_F)v)^2.
\end{eqnarray}

 Notice that $u=(I+\tilde{D}_F)v$ and $v=(I-D_F)u$, where $D_F u=\sum_{k=1}^ND_{\gamma_k}F_k(u)$, we see that $(I-\tilde{D}_F)^{-1}=I-D_F$. Denote by $\mathcal{A}_F: D(\mathcal{A})\rightarrow L^2(0,1)$ the operator
\begin{eqnarray}\label{e306}
&&\mathcal{A}_F v=(I-D_F)\nonumber\\&&[\mathcal{A}v+2\sum_{k,i=1}^N\lambda_k\langle D_{\gamma_i}\tilde{F}_i(v), \phi_k\rangle\phi_k-\sum_{i=1}^N\gamma_iD_{\gamma_i}\tilde{F}_iv].
\end{eqnarray}
Then, equation (\ref{e305}) can be equivalently written as
\begin{equation}\label{e307}
\frac{dv}{dt}+\mathcal{A}_F v= -\beta(I-D_F)((I+\tilde{D}_F)v)^2
\end{equation}
We set $\mathcal{G}(v)=-\beta(I-D_F)((I+\tilde{D}_F)v)^2=-\beta v(I+\tilde{D}_F)v$, then
\begin{equation}\label{e308}
v(t)=e^{-\mathcal{A}_Ft}v_0+\int_0^te^{-\mathcal{A}_F(t-s)}\mathcal{G}(v)(s)ds, t\geq 0.
\end{equation}
We are going to show firstly that, for $\|v_0\|$ sufficiently small, the map
$S: L^2(0,\infty; H_0^1(0,1))\rightarrow L^2(0,\infty; H_0^1(0,1))$, defined by
$$S(v)(t)=e^{-\mathcal{A}_Ft}v_0+\int_0^te^{-\mathcal{A}_F(t-s)}\mathcal{G}(v)(s)ds$$
is contraction on
$$K_r=\{v\in L^2(0,\infty; H_0^1(0,1)); \|v\|_{L^2(0,\infty; H_0^1(0,1))}\leq r\}$$
for some $r>0$.

By Theorem \ref{th201}, we know that $\mathcal{A}_F$ generates a strong continuous, and exponentially stable semigroup on $L^2(0,1)$, i.e.,
\begin{equation}\label{e309}
\|e^{-\mathcal{A}_F t}v_0\|\leq Ce^{-\mu t}\|v_0\|, \forall t>0.
\end{equation}
By inequalities (\ref{e221}) and (\ref{e223}), one can easily obtain the estimates
\begin{equation}\label{e226}
\int_0^{\infty}\lambda_j|v_j(t)|^2dt\leq C\|v_0\|^2, j=1, 2, \cdots,
\end{equation}
 where $v_j(t)=\langle v, \phi_j\rangle$. This implies that
\begin{equation}\label{e310}
\int_0^{\infty}\|e^{-\mathcal{A}_F t}v_0\|_1dt\leq C\|v_0\|, \forall v_0\in L^2(0,1).
\end{equation}
Moreover, sine $I+\tilde{D}_F$ is linear continuous on $L^2(0,1)$, it is easy to check that, for any $v, \bar{v}\in H_0^1(0,1)$,
$$\|\mathcal{G}(v)-\mathcal{G}(\bar{v})\|\leq C_1 \|v-\bar{v}\|_1(\|v\|_1+\|\bar{v}\|_1),$$
$$\|\mathcal{G}(v)\|\leq C_2\|v\|_1^2.$$
By the latter inequalities and (\ref{e309}), (\ref{e310}), adopt the method applied in the proof of Theorem 4.1 in \cite{02} (see also the proof of Theorem 5.1 in \cite{19}), we can obtain the existence of fixed point of map $S$, which turns out to be the unique solution to equation (\ref{e306}). Moreover, it is exponentially stable, i.e., $\|v(t)\|\leq C e^{-\gamma t}\|v(0)\|$, and the latter extends to the solution to (\ref{e101}), and (\ref{e302}) follows.
\end{proof}

\begin{proof}[Proof of Theorem 3.2]
Since $u_0\in I_C$, we see that $v_0\in H_0^1(0,1)$. By Theorem \ref{th301}, we know that
\begin{eqnarray}\label{e311}
&&\|v(t)\|+\|u(t)\|\leq Ce^{-\mu t}\|u_0\|, \forall t>0,\nonumber\\&& \int_0^{\infty}(\|u(t)\|_1^2+\|v(t)\|^2_1)dt\leq C.
\end{eqnarray}
One can refer from the expression of $R$ and the latter inequalities that
 \begin{equation}\label{e312}
\|R(t)\|\leq Ce^{-\mu t}\|u_0\|, \forall t>0.
\end{equation}
Denote by $\tilde{v}(x,t)$ the function $v(t,x)e^{\mu t/2 }$, it is not difficult to check that $\tilde{v}$ satisfies the following equation
\begin{equation}\label{e313}
 \left\{\begin{array}{ll}
\tilde{v}_t-\tilde{v}_{xx}-(\alpha+\mu/2) \tilde{v}=-\beta e^{\mu t/2}u^2+e^{\mu t/2}R , \mathrm{in} \ Q,\\
 \tilde{v}(0,t)=0, \tilde{v}(1,t)=0,\ \ \ \ \ \ t\in (0,\infty),\\
  \tilde{v}(x,0)=v_0(x)\ \ \ \ \ x\in(0,1).
\end{array}\right.
\end{equation}
Multiplying equation (\ref{e313}) by $\tilde{v}$ in the sense of $L^2(0,1)$, and integrating on $(0, t)$, we obtain by inequalities (\ref{e311}), (\ref{e312}), and the Sobolev imbedding Theorem that
\begin{equation}\label{e314}
\|\tilde{v}(t)\|^2+\int_0^{\infty}\|\tilde{v}_{x}(s)\|^2ds\leq C \|u_0\|^2, \forall t>0.
\end{equation}
Multiplying equation (\ref{e313}) by $-\tilde{v}_{xx}$, and integrating on $(0, t)$, it follows that
\begin{eqnarray}\label{e315}
&&\|\tilde{v}_{x}(t)\|^2+\int_0^{\infty}\|\tilde{v}_{xx}(s)\|^2ds\nonumber\\ &&\leq C (\int_0^{\infty}\|\tilde{v}_{x}(s)\|^2ds+\|u_0\|_1^2)
\leq C, \forall t>0.
\end{eqnarray}
The latter inequality implies that
\begin{equation}\label{e316}
\|v(t)\|_1^2\leq C e^{-\mu t}\|u_0\|_1^2, \forall t>0.
\end{equation}
The latter extends to the solution to (\ref{e101}), and (\ref{e304}) follows.
\end{proof}
\section{Numerical Examples}
\setcounter{equation}{0}
In this section, we present the numerical results of the Fisher's equation with feedback boundary feedback controller designed above. In model (\ref{e101}), we take $\alpha =30, \beta=0.30$, and set the initial profile as $u_0(x)=5xe^x, x\in [0,1]$.

 \begin{figure}[htbp]
  \centering
   %Requires \usepackage{graphicx}
  \includegraphics[width=0.3\textwidth]{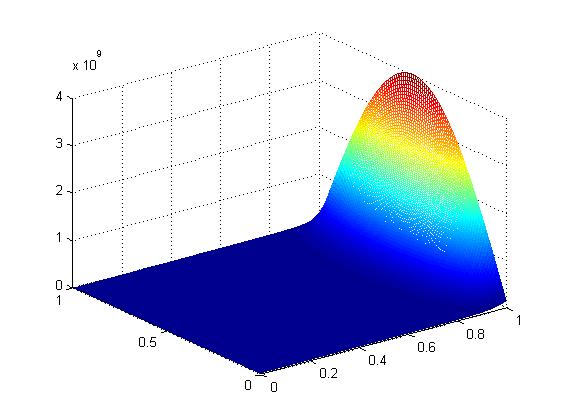}\\
  \caption{State of Fisher's equation without control}\label{Fig.1}
\end{figure}

We can observe from Figure \ref{Fig.1} that the state is unstable without control.

We take $\gamma_1=15, \gamma_2=20$ in (\ref{e203}), then
$$B_1=\pi^2\left(\begin{array}{ccc}\frac{1}{(45-\pi^2)^2}\ \ \ \ \frac{-2}{(45-\pi^2)(45-(2\pi)^2)}\\ \frac{-2}{(45-\pi^2)(45-(2\pi)^2)}\ \ \ \ \frac{4}{(45-(2\pi)^2)^2}\end{array}\right),$$
$$B_2=\pi^2\left(\begin{array}{ccc}\frac{1}{(50-\pi^2)^2}\ \ \ \ \frac{-2}{(50-\pi^2)(50-(2\pi)^2)}\\ \frac{-2}{(50-\pi^2)(50-(2\pi)^2)}\ \ \ \ \frac{4}{(50-(2\pi)^2)^2}\end{array}\right).$$
By Theorem \ref{th301}, we know that the control of feedback form
\begin{equation}\label{e403}
U(t)=F(u)(t):=\langle TB\left(\begin{array}{ccc}\int_0^1u\sin\pi xdx\\ \int_0^1u\sin2\pi xdx\end{array}\right), \left(\begin{array}{ccc}1\\ 1\end{array}\right)\rangle_2,
\end{equation}
exponentially stabilizes Fisher's equation (as shown in Figure \ref{Fig.2}).
Here
\begin{equation}\label{e401}
T:=\left(\begin{array}{ccc}\frac{-\pi}{45-\pi^2}\ \ \  \frac{2\pi}{45-(2\pi)^2}\\ \frac{-\pi}{50-\pi^2}\ \ \  \frac{2\pi}{50-(2\pi)^2}\end{array}\right), B=(B_1+B_2)^{-1}.
\end{equation}

 \begin{figure}[htbp]
  \centering
   %Requires \usepackage{graphicx}
  \includegraphics[width=0.3\textwidth]{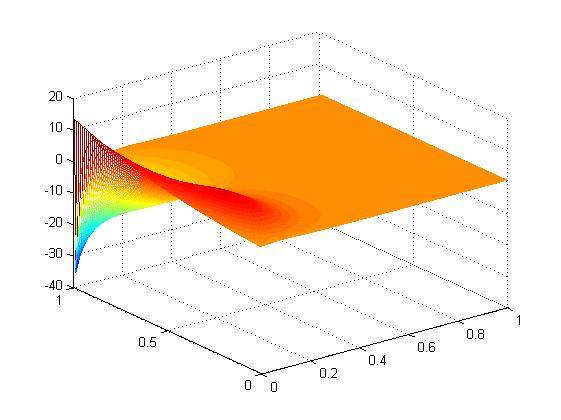}\\
  \caption{State of Fisher's equation with control}\label{Fig.2}
\end{figure}

As we can realize that it is more practical to use only part of information of the state, one can take a modified feedback form as
\begin{equation}\label{e404}
U(t)=F(u)(t):=\langle TB\left(\begin{array}{ccc}\int_a^bu\sin\pi xdx\\ \int_a^bu\sin2\pi xdx\end{array}\right), \left(\begin{array}{ccc}1\\ 1\end{array}\right)\rangle_2,
\end{equation}
where $[a,b]\subset [0,1]$ ia a proper set in $[0,1]$.

Fix $b=1$, and simulate the closed-loop system with different values for $a$. We find that for any $a\leq 0.24$ the system governed by Fisher's equation (\ref{e101}) can be still stabilized using the feedback form (\ref{e404}) (as shown in Figure \ref{Fig.3.1}). However, the value of $a$ can't be too small, as we have shown in Figure \ref{Fig.3.2}, when $a=0.25$, it seems that the system can't be stabilized with control of the form (\ref{e404}).
\begin{figure}[H]
\centering
\subfigure[State of Fisher's equation with $a=0.24$]{
\label{Fig.3.1}
\includegraphics[width=0.3\textwidth]{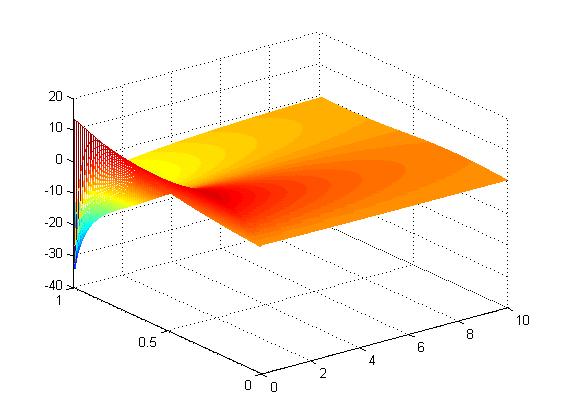}}
\subfigure[State of Fisher's equation with $a=0.25$]{
\label{Fig.3.2}
\includegraphics[width=0.3\textwidth]{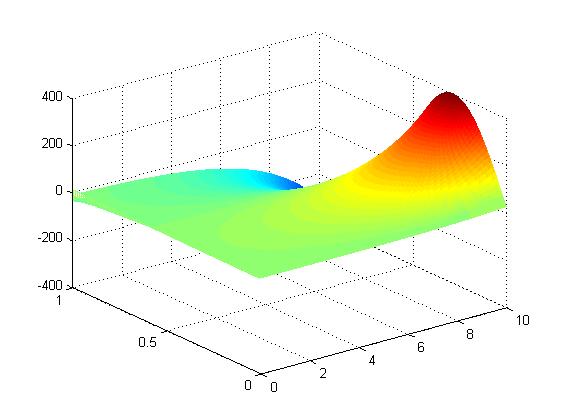}}
\caption{State of Fisher's equation with part information}
\label{Fig.3}
\end{figure}

As we have shown in Theorem \ref{th302}, the feedback controller $U(t)$ can stabilize the solution of Fisher's equation not only in $L^2(0,1)$, but also in $H^1(0,1)$, when the initial data satisfies some compatibility condition. Similarly as above, we take control of the form (\ref{e404}), which only use part of the state information. Simulating the closed-loop system with $b=1$, and with different values for $a$, we find that when $a\leq 0.24$, the solution is exponentially stable in $H^1(0,1)$, which is shown in Figure \ref{Fig.4}.
\begin{figure}[H]
\centering
\subfigure[$H^1$ norm of the state with different $a$]{
\label{Fig.4.1}
\includegraphics[width=0.4\textwidth]{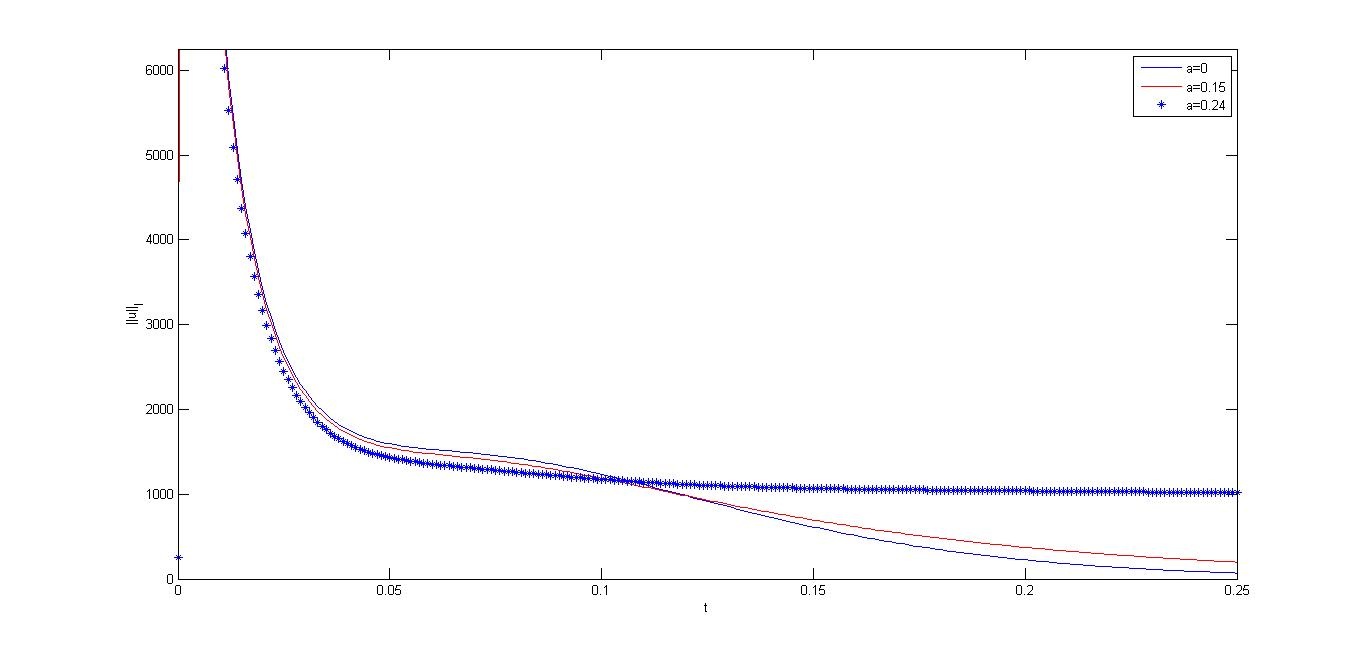}}
\subfigure[$H^1$ norm of the state with $a=0.24$]{
\label{Fig.4.2}
\includegraphics[width=0.4\textwidth]{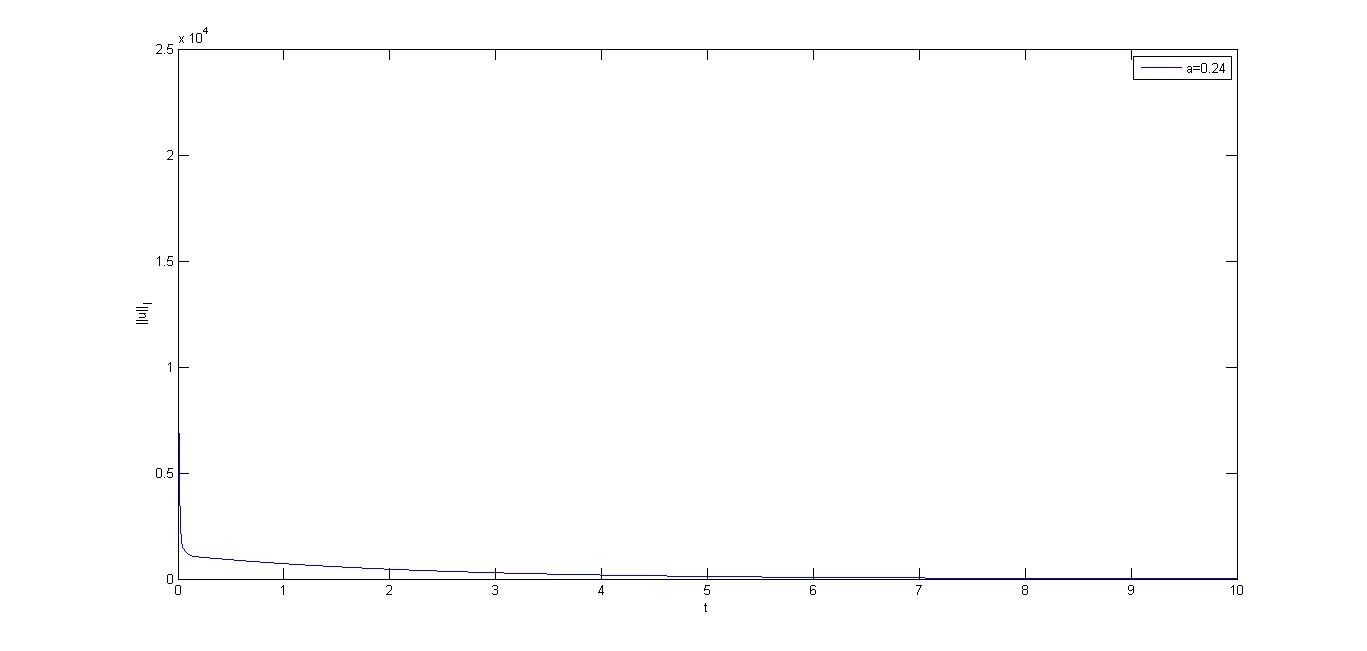}}
\caption{$H^1$ norm of the state of Fisher's equation with part information}
\label{Fig.4}
\end{figure}
It is reasonable that, when we use less information of the state for the feedback control, it is less effective of the feedback controller to stabilize the solution of the closed-loop system. We can see from Figure \ref{Fig.4} that the solution of the closed-loop system decay much slower when $a=0.24$ than the case $a=0$ and $a=0.15$.
\section{Appendix}
\setcounter{equation}{0}
We verify the claim that the sum $B_1+B_2+\cdots+B_N$ is invertible.
\begin{lemma}\label{lm501}
For any $\rho<\gamma_1<\gamma_2<\cdots<\gamma_N$, we have,
\begin{equation}\label{e501}
\left|\begin{array}{ccc}\frac{1}{\gamma_1-\lambda_1}\ \ \  \frac{1}{\gamma_1-\lambda_2}\    \ \ \cdots\ \ \   \frac{1}{\gamma_1-\lambda_N} \\ \frac{1}{\gamma_2-\lambda_1}\  \ \ \frac{1}{\gamma_2-\lambda_2}\   \ \  \cdots\  \ \  \frac{1}{\gamma_2-\lambda_N}\\ \cdots\  \ \ \  \cdots\  \ \ \  \cdots\  \ \ \  \cdots\\ \frac{1}{\gamma_N-\lambda_1}\  \ \ \frac{1}{\gamma_N-\lambda_2}\    \ \ \cdots\   \ \ \frac{1}{\gamma_N-\lambda_N}\end{array}\right|\neq 0.
\end{equation}
\end{lemma}
We refer the proof to (\cite{09}, Lemma 5.1).
\begin{lemma}\label{lm502}
The sum $B_1+B_2+\cdots+B_N$ is an invertible matrix, where $B_k, k=1, 2, \cdots, N$ are introduced in (\ref{e205}).
\end{lemma}
\begin{proof}
Applying by contradiction, let us assume that there is nonzero vector $z=(z_1,z_2,\cdots, z_N)^T\in R^N$, such that
$(B_1+B_2+\cdots+B_N)z=0$. It follows that
$$\sum_{k=1}^N\langle B_kz, z\rangle_N=0,$$
or, equivalently,
$$\sum_{i,k=1}^N[z_i\frac{1}{\gamma_k-\lambda_i}\phi_i'(1)]^2dx=0.$$
We deduce from above that
$$\sum_{i=1}^Nz_i\frac{1}{\gamma_k-\lambda_i}\phi_i'(1)=0, k=1, 2, \cdots, N.$$
For all $i=1,2,\cdots, N$, $\phi_i'(1)=\pi i\cos(\pi i)\neq 0$. The determinant of the matrix of the corresponding system is
$$\prod_{i=1}^N\phi_i'(1)\left|\begin{array}{ccc}\frac{1}{\gamma_1-\lambda_1}\ \ \  \frac{1}{\gamma_1-\lambda_2}\    \ \ \cdots\ \ \   \frac{1}{\gamma_1-\lambda_N} \\ \frac{1}{\gamma_2-\lambda_1}\  \ \ \frac{1}{\gamma_2-\lambda_2}\   \ \  \cdots\  \ \  \frac{1}{\gamma_2-\lambda_N}\\ \cdots\  \ \ \  \cdots\  \ \ \  \cdots\  \ \ \  \cdots\\ \frac{1}{\gamma_N-\lambda_1}\  \ \ \frac{1}{\gamma_N-\lambda_2}\    \ \ \cdots\   \ \ \frac{1}{\gamma_N-\lambda_N}\end{array}\right|\neq 0.$$
\end{proof}
Hence, it is necessary that $z=0$. This is a contradiction with our assumption.



%% References

%%\newpage

\bibliographystyle{plain}
%%\bibliography{Navier-Stokes}



\end{document}